\documentclass[12pt]{amsart}

\usepackage{graphicx}
\usepackage{amsfonts,graphics,amsmath,amscd,amssymb,latexsym,multicol,euscript}
\usepackage[all,cmtip,line]{xy}
\usepackage{overpic}
\usepackage[figurename=Figure, tablename=Figure, figurewithin=none, tablewithin=none, labelsep=space]{caption}

\newtheorem{theorem}{Theorem}[section]
\newtheorem{corollary}[theorem]{Corollary}
\newtheorem{definition}[theorem]{Definition}
\newtheorem{lemma}[theorem]{Lemma}
\newtheorem{proposition}[theorem]{Proposition}
\newtheorem{definition-lemma}[theorem]{Definition-Lemma}

\theoremstyle{definition}
\newtheorem{remark}[theorem]{Remark}
\newtheorem{question}[theorem]{Question}

\newcommand{\QED}{}

\newcommand\R{\mathbb{R}}
\newcommand\Q{\mathbb{Q}}
\newcommand\Z{\mathbb{Z}}

\newcommand\eps{\varepsilon}
\newcommand\bs{\bigskip}
\newcommand\ms{\medskip}

\newcommand\noi{\noindent}

\newcommand{\cn}{{}_{\mathrm{lcm}}}

\newcommand{\ol}[1]{\overline{#1}}

\DeclareMathOperator{\Null}{Null}

\DeclareMathOperator{\mult}{mult}
\DeclareMathOperator{\Int}{Int}

\DeclareMathOperator{\Supp}{Supp}

\DeclareMathOperator{\Bs}{Bs}
\DeclareMathOperator{\ra}{\rightarrow}

\DeclareMathOperator{\Exc}{Exc}

\DeclareMathOperator{\NE}{NE}
\DeclareMathOperator{\NM}{NM}

\DeclareMathOperator{\mov1}{\text{mov}^1}
\DeclareMathOperator{\bmov1}{\text{b-mov}^1}
\DeclareMathOperator{\bNM}{bNM}

\DeclareMathOperator{\Nu^1}{N^1}
\DeclareMathOperator{\Num1}{N_1}
\DeclareMathOperator{\Nef}{Nef}

\DeclareMathOperator{\Mob}{Mob}
\DeclareMathOperator{\Eff}{Eff}
\DeclareMathOperator{\Amp}{Amp}

\DeclareMathOperator{\Div}{Div}

\begin{document}

\title{On the dual of the mobile cone.}

\author{Sung Rak Choi}
\address{Department of Mathematics,
University of California, Riverside.
900 University Ave, Riverside, CA 92521}
\email{schoi@math.ucr.edu}

\maketitle

\begin{abstract}
We prove that the mobile cone and the cone
of curves birationally movable in codimension $1$ are dual to each other
in the $(K+B)$-negative part for a klt pair $(X,B)$.
This duality of the cones gives a partial answer to the problem posed by Sam Payne.
We also prove the cone theorem and the contraction theorem
for the expanded cone of curves birationally movable in codimension $1$.
\end{abstract}
\bigskip

\section{Introduction}
\label{section-intro}
Let $X$ be a normal projective algebraic variety defined over
an algebraically closed field $k$ (of characteristic $0$).
It is well-known that due to Kleiman-Seshadri, the cone of nef divisors
$\Nef(X)\subseteq \Nu^1(X)$ and the cone of curves (often called the Mori cone)
$\ol{\NE}(X)\subseteq\Num1(X)$ are dual to each other.
It is also well-known that due to Boucksom-Demaillly-Paun-Peternell \cite{bdpp},
the cone of pseudoeffective divisors $\ol\Eff(X)\subseteq\Nu^1(X)$
and the cone of movable curves $\ol\NM(X)\subseteq\Nu^1(X)$ are dual to each other:
$$
\xymatrix{
\ol{\Eff}(X)&\supseteq &\Nef(X)\\
\ol{\NM}(X)\ar@{<->}[u]^{\text{dual}}&\subseteq &\ol{\NE}(X)\ar@{<->}[u]^{\text{dual}}.
}
$$
The next most important cone in $\Nu^1(X)$ is probably the \emph{mobile cone} $\ol\Mob(X)$,
the closed convex cone spanned by all the numerical classes of mobile divisors.
Mobile divisors are the divisors whose $\R$-base loci (see Section \ref{section-movable cone}) are of codimension $\geq 2$.
The mobile cone $\ol\Mob(X)$ is a subcone of $\ol\Eff(X)$ which contains the nef cone $\Nef(X)$:
$\Nef(X)\subseteq\ol{\Mob}(X)\subseteq \ol\Eff(X)$.
It is natural to ask what the dual of the mobile cone $\ol{\Mob}(X)$ is.
In this paper, we will find a partial answer to this question.

A naive candidate for the dual of the mobile cone $\ol{\Mob}(X)$ is the
closed convex cone $\ol\NM^1(X)\subseteq\Num1(X)$ spanned by the classes of curves movable in
codimension $1$ subvarieties.
However, Payne in \cite{payne} showed that in general the cone $\ol\NM^1(X)$ is strictly smaller
than the dual $\ol\Mob(X)^{\vee}$.
He also showed that in the case of complete $\Q$-factorial toric varieties,
we have to also allow the classes of curves movable in codimension $1$ subvarieties
on \emph{$\Q$-factorial small modifications} of $X$
in order to obtain the correct dual of $\ol\Mob(X)$ (\cite[Theorem 2]{payne}).
Following his ideas, we will give a partial generalization of his result
for $\Q$-factorial klt pairs $(X,B)$ where $X$ is not necessarily toric, which is valid in
the $(K+B)$-negative part of the cone (Theorem \ref{thrm-mob bnm dual}).
This also gives a partial answer to the problem posed
in \cite{payne} for $\Q$-factorial Fano type varieties (Corollary \ref{cor-FT mov}).

Let $f:X\dashrightarrow X'$ be a small birational map between $\Q$-factorial normal projective varieties.
Since it is known that $\Nu^1(X)$ and $\Nu^1(X')$ are isomorphic
under $f_*$ \cite[12-2-1]{matsuki},
their dual spaces $\Num1(X)$ and $\Num1(X')$, respectively,  are also isomorphic:
$\Num1(X)\cong\Num1(X')$.
Under this isomorphism, a class $\alpha=[C]\in\Num1(X')$ defined by
a $\mov1$(\emph{movable in codimension $1$})-curve $C$ on $X'$ can be pulled back to $\Num1(X)$
and we can simply consider $\alpha$ as a class in $\Num1(X)$.
The $\mov1$-curve $C$ on $X'$ is called a
$\bmov1$(\emph{birationally movable in codimension $1$})-curve of $X$.
We define $\ol\bNM^1(X)$ as the closed convex cone in $\Num1(X)$
spanned by all the classes of $\bmov1$-curves of $X$.
See Section \ref{section-movable cone} for details.

We have the following partial duality result.

\begin{theorem}\label{thrm-mob bnm dual}
Let $(X,B)$ be a $\Q$-factorial klt pair. Then
$$
\ol\NE(X)_{K+B\geq 0}+\ol{\Mob}(X)^{\vee}=
\ol\NE(X)_{K+B\geq 0}+\ol{\bNM}^1(X).
$$
\end{theorem}
In other words, the dual cone $\ol\Mob(X)^{\vee}$ coincides with $\ol\bNM^1(X)$
at least in some portion of the $(K+B)$-negative part.
Inspired by the results in \cite{araujo}
and \cite{leh}, we also prove the following cone theorem for $\ol\bNM^1(X)$.

\begin{theorem}[The Cone Theorem for $\ol{\bNM}^1(X)$]\label{thrm-cone thrm NM^1}
Let $(X,B)$ be a $\Q$-factorial klt pair.
Then there exists a countable set of $\bmov1$-curves $\{C_i\}_{i\in I}$ of $X$ such that
$$
\ol\NE(X)_{K+B\geq 0}+\ol{\bNM}^1(X)=
\ol\NE(X)_{K+B\geq 0}+\ol{\sum_{i\in I} \R_{\geq 0} \cdot [C_i]}
$$
and for any ample $H$ and any $\varepsilon>0$,
there exists a finite subset $J\subseteq I$ such that
$$
\ol\NE(X)_{K+B+\eps H\geq 0}+\ol{\bNM}^1(X)=
\ol\NE(X)_{K+B+\eps H\geq 0}+\sum_{j\in J} \R_{\geq 0} \cdot [C_j].
$$
The rays $\{R_i=\R_{\geq 0}[C_i]\}_{i\in I}$ in the first equality can accumulate only at the hyperplanes
supporting both $\ol\NE(X)_{K+B\geq 0}$ and $\ol{\bNM}^1(X)$.
\end{theorem}

Note that this is actually a structure theorem for the expanded cone
$\ol\NE(X)_{K+B\geq 0}+\ol\bNM^1(X)$ (see Figure \ref{figure} in Section \ref{section-cone theorem}).
We also prove the following contraction theorem for $\ol\bNM^1(X)$.
We call an extremal ray $R$ of $\ol{\bNM}^1(X)$ a \emph{$\mov1$-co-extremal ray for} $(X,B)$
if it is $(K+B)$-negative and it is also an extremal ray for the expanded cone
$\ol{\NE}(X)_{K+B\geq 0}+\ol{\bNM}^1(X)$.
See Section \ref{section-cone theorem} for details.

\begin{theorem}[Contraction Theorem for $\mov1$-co-extremal rays]\label{thrm-contraction thrm NM^1}
Let $(X,B)$ be a $\Q$-factorial klt pair.
Let $R$ be a $\mov1$-co-extremal ray of $\ol{\bNM}^1(X)$ for $(X,B)$.
Then the following hold:
\begin{enumerate}
\item[(1)] there exists a small birational map $\varphi:X\dashrightarrow X'$
and a contraction
$\psi:X'\ra Y$ which is either a divisorial contraction or a Mori fiber space
such that the $\mov1$-co-extremal ray $R$ is spanned by a $\mov1$-curve $C$ on $X'$ if and only if
$C$ is contracted by $\psi$, and

\item[(2)] the composition map $\psi\circ\varphi$ is uniquely determined by $R$.
\end{enumerate}
\end{theorem}

\bs

This paper is organized as follows:

In section \ref{section-pre}, we review the definitions and properties of
the non-ample locus $\mathbf B_+(D)$ and non-nef locus $\mathbf B_-(D)$
of divisors $D$.
We also recall some necessary results from the theory of log minimal model program (LMMP).
In section \ref{section-movable cone}, we study the structure of the mobile cone $\ol\Mob(X)$.
The proof of Theorem \ref{thrm-mob bnm dual} is given in this section.
In section \ref{section-cone theorem}, we prove Theorem \ref{thrm-cone thrm NM^1}
and Theorem \ref{thrm-contraction thrm NM^1}.

\bs

I would like to thank S. Boucksom for pointing out an error in the preliminary version of this paper
and Z. Ran for answering my question.
I also would like to thank V. Shokurov for the encouragement.

\section{Preliminaries}\label{section-pre}

Let $X$ be a normal projective variety.
For a $\Z$-divisor $D$ on $X$, its base locus $\Bs(D)$ is defined as
the support of the intersection of the elements in the usual $\Z$-linear system
$|D|=\{D'\in\Div_\Z(X)\;|\; D\sim D'\geq 0\}$.
For a $\Q$-divisor $D$, the \emph{stable base locus} of $D$ is defined as
$\mathbf{B}(D):=\bigcap_m \Bs(mD)$ where
the intersection is taken over the positive integers $m$ such that
$mD$ are integral.
For an $\R$-divisor $D$ on $X$, the $\R$-linear system is defined as
$|D|_{\R}:=\{D'\in\Div_\R(X)\;|\;D\sim_{\R}D'\geq 0\}$ and its \emph{$\R$-stable base locus} as
$\mathbf{B}_{\R}(D):=(\cap|D|_{\R})_\text{red}$.
Clearly, $\mathbf{B}_{\R}(D)\subseteq\mathbf{B}(D)$ for a $\Q$-divisor $D$.
From now on, we always use $\R$-divisors unless otherwise stated.
\bs

For a divisor $D$ on $X$, we define the \emph{non-ample locus}
(or \emph{augmented base locus}) of $D$ as
$$
\mathbf{B}_+(D):=\bigcap_{\text{ample\;} A}\mathbf{B}(D- A)
$$
where the intersection is taken over all ample divisors $A$ such that $D-A$ are $\Q$-divisors.
Note that $\mathbf{B}_+(D)=X$ if $D$ is not big.
As the name suggests, $D$ is ample if and only if $\mathbf B_+(D)=\emptyset$.
We define the \emph{non-nef locus}
(or \emph{diminished base locus}) of $D$ as
$$
\mathbf{B}_-(D):=\bigcup_{\text{ample\;} A}\mathbf{B}(D+A)
$$
where the union is taken over all ample divisors $A$ such that $D+A$ are $\Q$-divisors.
Note that $\mathbf B_-(D)=X$ if $D$ is not pseudoeffective.
It is easy to see that $D$ is nef if and only if $\mathbf{B}_-(D)=\emptyset$.
It is known that $\mathbf{B}_+(D)=\mathbf{B}_+(D-A)=\mathbf B_-(D-A)$
for any sufficiently small ample divisor $A$ \cite[Proposition 1.21]{elmnp2}.
It is also well known that the base loci $\mathbf B_+(D),\mathbf B_-(D)$ depend only on the numerical
class of $D$ whereas $\mathbf B(D)$ is not in general.
The non-ample locus $\mathbf B_+(D)$ is Zariski closed for any $D$
whereas the non-nef locus $\mathbf B_-(D)$ is \emph{a priori} not in general.

\bs

\begin{remark}\label{remk-B_- not Zar closed}
The following inclusions are easy to verify and often useful:
for any ample divisor $A$,
$$
\mathbf B_+(D+A)\subseteq\mathbf B_-(D),\;\; \mathbf B(D+A)\subseteq\mathbf B_-(D),\;\text{and}\;\;
\mathbf B_+(D+A)\subseteq\mathbf B(D).
$$
It is easy to see that if $V\subseteq\mathbf B_-(D)$ for some subvariety $V$ of $X$,
then there exists a small ample divisor $A$ such that $V\subseteq\mathbf B_+(D+A)$.
Thus we can also define the non-nef locus as $\mathbf B_-(D):=\cup\mathbf B_+(D+A)$
where the union is taken over all ample divisors $A$.

\end{remark}

See \cite{bbp},\cite{elmnp},\cite{elmnp2},\cite{laza2} for more details about the non-ample loci
and non-nef loci.

\bs

A big divisor $D$ is \emph{$\R$-mobile} if $\mathbf{B}_\R(D)$ does not have a divisorial component.
We define the cones in the numerical space $\Nu^1(X)$:
$$
\begin{array}{cl}
\Amp(X)&:=\{[D]\in\Nu^1(X)\;|\; \text{$D$ is ample}\;\},\\
\Mob(X)&:=\{[D]\in\Nu^1(X)\;|\;\text{$D\equiv D'$ for some $\R$-mobile $D'$}\},\\
\Eff(X)&:=\{[D]\in\Nu^1(X)\;|\;\text{$D\equiv D'$ for some effective $D'$}\}.
\end{array}
$$
Their closures $\Nef(X)=\ol{\Amp}(X)$, $\ol{\Mob}(X)$, and $\ol{\Eff}(X)$ are
called the \emph{nef cone, mobile cone}, and \emph{pseudoeffective cone}, respectively.
They satisfy the inclusion:
$\Nef(X)\subseteq \ol{\Mob}(X)\subseteq\ol{\Eff}(X)$.
We will study the mobile cone $\ol\Mob(X)$ in detail using the base loci $\mathbf B_-, \mathbf B_+$
in Section \ref{section-movable cone}.

For a cone $V\subseteq \Num1(X)$, a divisor $D$ and $\square\in\{=,<, >,\geq,\leq\}$, we define
$$
V_{D \square\;0}:=V\cap \{C\in\Num1(X)\;|\;D\cdot C \;\square\; 0\}.
$$
An \emph{extremal face} $F$ of a closed convex cone $V$ satisfies the two conditions
1) $F$ is a convex subcone of $V$, and
2) if $v+u\in F$ for $u,v\in V$, then $u,v\in F$. A one dimensional extremal face is called an
\emph{extremal ray}.

\bs

We use the standard notions of singularities of pairs $(X,B)$ in the log minimal model
 program (LMMP, for short) \cite{komo},\cite{isksh}.
We briefly recall the basics of the LMMP.
For an exceptional prime divisor $E$ over $X$, $a(E,X,B)$ denotes the log discrepancy of
$(X,B)$ at $E$.
\begin{definition}
Let $(X,B)$ be a $\Q$-factorial klt pair and let $\varphi:X\dashrightarrow Y$ be a birational map
to a projective $\Q$-factorial variety $Y$. Let $B_Y:=f_*B$.
\begin{enumerate}
\item[(1)] A pair $(Y,B_Y)$ is called a \emph{log terminal model of} $(X,B)$ if the pair $(Y,B_Y)$ is klt, $K_Y+B_Y$ is
nef, and the inequality $1-\mult_EB<a(E,Y,B_Y)$ holds for any $\varphi$-exceptional prime divisor $E$.

\item[(2)] A pair $(Y,B_Y)$ equipped with a fibration $Y\rightarrow T$ is called a
\emph{Mori log  fibration of} $(X,B)$ if $(Y,B_Y)$ is klt, $\dim T<\dim Y$, the relative Picard number
$\rho(Y/T)=1$, and $-(K_Y+B_Y)$ is ample over $T$.

\item[(3)] A \emph{resulting model of} $(X,B)$ is either a log terminal model (1) or a Mori log fibration (2).
\end{enumerate}
\end{definition}

By the LMMP,  any  $\Q$-factorial klt pair $(X,B)$ is expected to have a resulting model and
it cannot have both resulting models simultaneously \cite[2.4.1]{3fold-logmod}.
It is known that $(X,B)$ has a log terminal model as its resulting model
if and only if $K+B$ is pseudoeffective \cite{bchm}.

\bs

We will use the following lemma often.
\begin{lemma}\label{lem-ample klt}\emph{\cite[Example 9.2.29]{laza2}}
Let $(X,B)$ be a klt pair and $H$ be an ample divisor on $X$. Then there exists an effective divisor
$H'\sim_\R H$ such that $(X,B+H')$ is klt.
\end{lemma}

If the pairs $(X,B), (X,B')$ are klt and $B\sim_\R B'$, then $(X,B)$ and $(X,B')$ have the same resulting models
by the LMMP. Therefore by Lemma \ref{lem-ample klt}, given a klt pair $(X,B)$ and an ample divisor $H$,
we may assume that $(X,B+H)$ is klt in order to run the LMMP or
to study the resulting models of $(X,B+H)$.

\bs

We review some necessary results from \cite{bchm}.
First, we state an important result about the decomposition of the following set:
$$
\mathcal{E}_H:=\{B\in U|\;B\geq 0, \;\text{$(X,B)$ is klt and $K+B+H$ is pseudo-effective}\;\}
$$
where $U$ is a finite dimensional subspace of real Weil divisors which is defined
over the rationals and $H$ is an ample divisor on $X$.

\begin{theorem}\label{thrm-bchm decomposition}\emph{\cite[Corollary 1.1.5]{bchm}}
Let $H$ be a rational ample divisor and
suppose that for some $B_0\in\mathcal E_H$, the pair $(X,B_0)$ is klt.
Then for any $B\in\mathcal{E}_H$, the pair $(X,B)$ has a  log terminal model. Furthermore,
there exist finitely many birational maps $\varphi_i:X\dashrightarrow X_i$ ($1\leq i\leq p$)
and the set $\mathcal E_H$ is decomposed into finitely many rational polytopes $\mathcal W_i$
$$
\mathcal E_H=\bigcup_i^p\mathcal W_i,
$$
satisfying the following condition:
if, for $B\in\mathcal E_H$, there exists a birational contraction $\varphi:X\dashrightarrow Y$
which is a log terminal model of $(X,B)$, then $\varphi=\varphi_i$ for some $1\leq i \leq p$.
\end{theorem}

In \cite{chsh}, the similar decomposition problem
(which is called the \emph{geography}) is also studied in detail in terms of b-divisors.
The following is also the consequence of \cite{bchm}.

\begin{theorem}\label{thrm-from bchm}\emph{\cite{bchm}}
Let $X$ be a projective variety and let $D:=K+B$ where
$(X,B)$ is klt, $B$ is big and $K+B$ is pseudoeffective. Then there exists a birational map $\Phi:X\dashrightarrow X'$ such that
\begin{enumerate}
\item $D':=\Phi_*D$ is nef,
\item $D\geq D'$ (i.e. the inequality is satisfied after pulling back to the graph),
\item $\Phi$ is surjective in codimension $1$, and
\item a prime divisor $E$ of $X$ is contracted by $\Phi$ if and only if it is a divisorial component
of $\mathbf B_-(D)$.
\end{enumerate}
Furthermore, if $K+B$ is big, then there exists a contraction $\Psi:X'\rightarrow Y$,
(which is the unique log canonical model of $(X,B)$) where $D'$ vanishes on the curves contracted by $\Psi$.
\end{theorem}

Condition \emph{4} is a well-known reformulation of the ``strict negativity'' condition of \cite{bchm}
by Kawamata \cite[Lemma 2]{kawa-remarks on div cone}.
The prime divisors supported in the \emph{numerically fixed part} of $D$  in
\cite{kawa-remarks on div cone}
coincide with the divisorial components of $\mathbf B_-(D)$.
By conditions \emph{3} and \emph{4},
the map $\Phi$ is an isomorphism in codimension $1$if $D=K+B\in\ol\Mob(X)$.

\section{The mobile cone $\ol\Mob(X)$}\label{section-movable cone}

We will characterize the mobile cone $\ol\Mob(X)$ using the non-ample locus $\mathbf B_+$
and non-nef locus $\mathbf B_-$. We will also study the dual of $\ol\Mob(X)$.

\ms

Let $\Mob_+(X)$ (resp. $\Mob_-(X)$) be the cone in $\Nu^1(X)$ spanned by the divisors $D$
such that $\mathbf{B}_+(D)$ (resp. $\mathbf{B}_-(D)$) do not contain codimension $1$ subvarieties.
It is easy to see that $\mathbf{B}_-(D)\subseteq\mathbf{B}_{\R}(D)\subseteq\mathbf{B}_+(D)$.
Therefore $\Mob_+(X)\subseteq\Mob(X)\subseteq\Mob_-(X)$.

\begin{lemma}\label{lem-same mobile}
Let $\Mob_-(X),\Mob_+(X)$ be the cones in $\Nu^1(X)$ defined above.
\begin{enumerate}
\item Let $D\in\ol{\Mob_+}(X)$. Then $D+A\in\Mob_+(X)$ for any ample divisor $A$.

\item The cone $\Mob_+(X)$ is open and $\overline{\Mob_+}(X)=\Mob_-(X)$.
In particular, the cone $\Mob_-(X)$ is closed and
$$\ol{\Mob_+}(X)=\ol\Mob(X)=\Mob_-(X).$$
Furthermore, $\Int\Mob_-(X)=\Mob_+(X)$.
\end{enumerate}
\end{lemma}

\begin{proof}
(1) Since $D\in\ol{\Mob_+}(X)$, there exists a sequence $D_i\in\Mob_+(X)$
such that $D_i\to D$ as $i\to \infty$.
For a fixed ample divisor $A$, by taking $i$ sufficiently large,
we may assume that $A-(D_i-D)$ is ample for all $i$.
Thus $\mathbf B_+(D_i)\supseteq\mathbf B_+(D_i+A-(D_i-D))=\mathbf B_+(D+A)$
and $D+A\in\Mob_+(X)$.

(2) The openness of the cone $\Mob_+(X)$ follows from \cite[Corollary 1.6]{elmnp2}:
there exists a small open neighborhood $N$ of $D$ such that for any $D'\in N$,
$\mathbf B_+(D')\subseteq\mathbf B_+(D)$.

To prove the inclusion $\overline{\Mob_+}(X)\subseteq\Mob_-(X)$,
let $D\in\ol{\Mob_+}(X)$.
If $D\not\in\Mob_-(X)$, then there exists a codimension $1$ subvariety $E$ of $X$ such that $E\subseteq\mathbf B_-(D)$.
By remark \ref{remk-B_- not Zar closed}, $E\subseteq\mathbf B_+(D+A)$ for some ample divisor $A$,
but it is a contradiction to (1).
The inclusion $\overline{\Mob_+}(X)\supseteq\Mob_-(X)$ can be seen as follows.
Let $D\in\Mob_-(X)$. Then for a fixed ample divisor $A$, $\{D+\frac 1m A\}$ is a sequence
in $\Mob_+(X)$ since $\mathbf B_+(D+\frac 1m A)\subseteq\mathbf B_-(D)$
(Remark \ref{remk-B_- not Zar closed}).
Thus the limit of the sequence must belong to $\ol{\Mob_+}(X)$,
i.e., $D\in\ol{\Mob_+}(X)$.
\QED
\end{proof}

\begin{lemma}\label{lem-D in baoundary Mob}
Let $D$ be a big divisor such that $D\in\partial\ol{\Mob}(X)$.
Then there exists a divisorial component $E\subseteq\mathbf B_+(D)$ such that
$E\not\subseteq\mathbf B_-(D)$.
In particular, for any ample divisor $A$ on $X$,
$D+A$ is $\R$-mobile and $D-A$ is not $\R$-mobile.
\end{lemma}
\begin{proof}
If $D\in\partial\ol{\Mob}(X)$, then by Lemma \ref{lem-same mobile} $D\not\in\Mob_+(X)$ and
$D\in\Mob_-(X)$. Thus there exists a subvariety $E$ of codimension $1$ such that
$E\subseteq\mathbf B_+(D)$ and $E\not\subseteq\mathbf B_-(D)$.
Since $D$ is big, $E$ is an irreducible component of $\mathbf B_+(D)$.
It is easy to see that $D+A$ is $\R$-mobile by the definition of $\mathbf B_-(D)$ and
that $D-A$ is not $\R$-mobile since $\mathbf B_+(D)\subseteq\mathbf B_{\R}(D-A)$.
\QED
\end{proof}

\begin{proposition}\label{prop-component under small}
Let $f:X\dashrightarrow X'$ be a small birational map between  projective $\Q$-factorial varieties.
Suppose that for a big divisor $D$ on $X$,
there exists a divisorial component $E\subseteq\mathbf B_+(D)$.
Then $E':=f_*{E}$ is also a divisorial component of $\mathbf B_+(D')$ where
$D'=f_*D$.
\end{proposition}

\begin{proof}
Let $W$ be a common resolution of $X$ and $X'$ with $p:W\to X$ and $q:W\to X'$.
By Proposition 1.5 of \cite{bbp}, we have
$$
\begin{array}{rl}
\mathbf B_+(p^*(D))&=p^{-1}(\mathbf B_+(D))\cup\Exc(p),\\
\mathbf B_+(q^*(D'))&=q^{-1}(\mathbf B_+(D'))\cup\Exc(q),
\end{array}
$$
and $\mathbf B_+(p^*(D))=\mathbf B_+(q^*(D'))$.
If $E\subseteq\mathbf B_+(D)$ is a divisorial component, then $E_W:=p^{-1}_*E$
is a divisorial component of $\mathbf B_+(p^*(D))$.
The divisor $E_W$ is not $q$-exceptional
because $X$ is isomorphic to $X'$ in codimension $1$.
Thus $E'=q_*(E_W)$ is also a divisorial component of $\mathbf B_+(D')$. \QED
\end{proof}

Nakamaye gave another characterization of the non-ample locus $\mathbf B_+(D)$
when $D$ is nef.
We define the \emph{null locus} $\Null(D)$ of a nef and big divisor $D$ as
$\Null(D):=\bigcup_V\{V\subseteq X\;|\;D^k\cdot V=0 \text{ where } \dim V=k>0\}$.
\begin{theorem}[Nakamaye's theorem]\label{thrm-nakamaye}
\emph{\cite[Theorem 10.3.5]{laza2}, \cite{nakamaye}}
Let $D$ be a nef and big divisor on $X$. Then
$$
\mathbf B_+(D)=\Null(D).
$$
\end{theorem}
This implies $D^{\dim V}\cdot V=0$ for any irreducible component $V$ of $\mathbf B_+(D)$.
We will also need the following result due to Khovanskii and Teissier.
\begin{theorem}[Khovanskii-Teissier inequality]\label{thrm-hodge ineq}\emph{\cite[Theorem 1.6.1]{laza1}}
Let $X$ be a variety of dimension $d$ and $D_i$ be nef divisors. Then
$$
D_1\cdot D_2\cdots D_d\geq (D_1^d)^{\frac{1}{d}}\cdot (D_2^d)^{\frac{1}{d}}\cdots
(D_d^d)^{\frac{1}{d}}.
$$
\end{theorem}

\bs

Taking into consideration of Payne's idea \cite{payne}, we define the following curves.
\begin{definition}\label{def-mov mov1 curve}
Let $X$ be a $\Q$-factorial normal variety of dimension $d$.
\begin{itemize}
\item A curve $C$ on $X$ is called a \emph{movable curve} if it
is a member of a family of curves covering $X$.
\item A curve $C$ on $X$ is called a \emph{$\mov1$(movable in codimension $1$)-curve} if
it is a member of a family of curves
covering a subvariety of codimension $1$.
\item A $\mov1$-curve $C$ on some $\Q$-factorial $X'$ which is isomorphic to $X$
in codimension $1$ is called a \emph{$\bmov1$(birationally movable in codimension $1$)-curve of} $X$.
\end{itemize}
\end{definition}
Note that as explained in Introduction, a $\bmov1$-curve $C$ of $X$
defines a class $\alpha=[C]\in\Num1(X)$
even though the curve $C$ may not be defined on $X$.
Thus we may treat a $\bmov1$-curve $C$ as a class in $\Num1(X)$.
We let $\NM(X)$, $\NM^1(X)$ be the cones in $\Num1(X)$ that are
spanned by the classes of movable curves and $\mov1$-curves on $X$, respectively.
We define $\NM^1(X,X')$ as the image in $\Num1(X)$ of the cone $\NM(X')$
under the isomorphism $\Num1(X')\cong\Num1(X)$.
Lastly, we define $\bNM^1(X)$ as the cone in $\Num1(X)$ spanned by $\bmov1$-curves of $X$.
It is easy to see that
$$
\ol{\bNM}^1(X)=\ol{\sum_{X\dashrightarrow X'}\ol{\NM}^1(X,X')}\;,
$$
where the summation is taken over all $\Q$-factorial $X'$ that are isomorphic to $X$
in codimension $1$.
By definition, a movable curve is $\mov1$ and a $\mov1$-curve is a $\bmov1$-curve of $X$.
Thus
$$
\ol\NM(X)\subseteq\ol{\NM}^1(X)\subseteq\ol{\bNM}^1(X).
$$
It is important to note that the inclusion on the right is strict in general
by Payne's counterexample \cite[Example 1]{payne}.

\ms

\begin{theorem}\label{thrm-NM Eff dual} The following hold:
\begin{enumerate}
\item[\emph{(1)}] $\Nef(X)=\ol\NE(X)^{\vee}$.
\item[\emph{(2)}] $\ol\Eff(X)=\ol\NM(X)^{\vee}$.
\end{enumerate}
\end{theorem}
\begin{proof}
(1) It is a well known result due to Kleiman-Seshadri. See \cite[Proposition 1.4.28]{laza1}.
(2) It is the main result of \cite{bdpp}  for smooth varieties.
The result also holds for $\Q$-factorial varieties.\QED
\end{proof}

According to Theorem \ref{thrm-mob bnm dual},
the cones $\ol{\Mob}(X)$ and $\ol{\bNM}^1(X)$ are dual to each other
at least in some part of the cones.
In order to prove Theorem \ref{thrm-mob bnm dual}, we prove the following equivalent dual statement:

$$
(*)\;\left. \begin{array}{l}
\text{the cones $\ol\Mob(X)$ and $\ol{\bNM}^1(X)^{\vee}$
coincide inside the convex cone}\\
\text{$P=\Nef(X)+\R_{\geq 0}\cdot [K+B]$.}
\end{array}\right.
$$

We start with an easy observation.
\begin{lemma}\label{lemma-nonneg intersect}
We have the following nonnegative intersection pairing:
$$
(\alpha,\beta)\in  \ol{\Mob}(X)\times \ol{\bNM}^1(X)\longmapsto \alpha\cdot\beta\geq 0.
$$
\end{lemma}
\begin{proof}
Let $D$ be an $\R$-mobile divisor and $C$ be a $\bmov1$-curve on $X$. Since the numerical
classes in $\Num1(X)$ are preserved under a small birational map, we may assume that $C$ is a
$\mov1$-curve on $X$.
Then since $C$ moves in a family of curves covering a subvariety of codimension $1$,
we may assume that $C$ is disjoint from the base locus of $D$ which is of codimension$\geq 2$.
Thus $C\cdot D\geq 0$.
The classes $\alpha$ and $\beta$ are the limits of the classes of such curve $C$ and divisor $D$.
Therefore $\alpha\cdot \beta\geq 0$ by continuity.\QED
\end{proof}

\begin{proof}[Proof of Theorem \ref{thrm-mob bnm dual}]

\noi (Step 1)
As we stated above, we prove the dual statement $(*)$.
By Lemma \ref{lemma-nonneg intersect}, we have $\ol\Mob(X)\subseteq\ol\bNM^1(X)^{\vee}$.
This in particular implies
$$\ol\Mob(X)\cap P\subseteq
\ol\bNM^1(X)^{\vee}\cap P,$$
where $P= \ol\Nef(X)+\R_{\geq 0}\cdot [K+B]$.
Suppose that the strict inclusion $\subsetneq$ holds.
Note that since $\ol{\bNM}^1(X)\supseteq\ol\NM(X)$,
we have $\ol{\bNM}^1(X)^{\vee}\subseteq\ol\Eff(X)$ by (2) of Theorem \ref{thrm-NM Eff dual}.
Note also that $\ol{\bNM}^1(X)^{\vee}=\bigcap\ol\NM^1(X,X')^{\vee}$, where the
intersection is taken over all $\Q$-factorial $X'$ that are isomorphic to $X$ in codimension $1$.
Therefore, if the inclusion above is strict,
then there exists a big divisor $D\in\partial\ol{\Mob}(X)\cap \Int P$
and $D\in\Int(\ol\NM^1(X,X')^{\vee})$ for any $\Q$-factorial $X'$ which is isomorphic to $X$ in codimension $1$.

\smallskip

\noi (Step 2)
There exists an ample divisor $H$ such that $rD\equiv K+B+H$ for some $r>0$.
By rescaling, we may assume that $D\equiv K+B+H$.
By Lemma \ref{lem-ample klt}, we may assume that $(X,B+H)$ is klt.
By Theorem \ref{thrm-from bchm}, there exists a log terminal model $f:X\dashrightarrow Y$
of $(X,B+H)$ which is an isomorphism in codimension $1$.

\smallskip

\noi(Step 3)
Since $D\in\partial\ol\Mob(X)$ and $D$ is big, there exists a divisorial component $E\subseteq\mathbf B_+(D)$
(Lemma \ref{lem-D in baoundary Mob})
and since the modification $f$ is small, $E_Y:=f_*E$ is also a divisorial component of $\mathbf B_+(D_Y)$
(Proposition \ref{prop-component under small}).
This implies that $D_Y\not\in\Int\ol\Mob(Y)$.
However, from Step 1, we have $D_Y\in\Int\big(\ol\NM^1(Y)^{\vee}\big)$.
Since $D_Y$ is also nef, by Lemma \ref{lem-lem for main thrm} we must have $D_Y\in\Int\ol\Mob(Y)$,
and this is a contradiction. \QED
\end{proof}

The following lemma will finish the above proof.

\begin{lemma}\label{lem-lem for main thrm}
If $X$ is a  projective  $\Q$-factorial variety of dimension $n$, then we have
$$
\Nef(X)\cap\Int\big(\ol\NM^1(X)^{\vee}\big)\subseteq\Int\ol\Mob(X).
$$
\end{lemma}
\begin{proof}
Let $D\in\Nef(X)\cap\Int(\ol\NM^1(X)^{\vee})$. Note that $D$ is big by \cite{bdpp}
((2) of Theorem \ref{thrm-NM Eff dual}).
If $D$ does not belong to $\Int\ol\Mob(X)$, then there exists a divisorial component $E\subseteq\mathbf B_+(D)$
and by Nakamaye's theorem (Theorem \ref{thrm-nakamaye}), $D^{n-1}\cdot E=0$.
Since $D\in\Int(\ol\NM^1(X)^{\vee})$, there exists some ample $\Q$-divisor $A$ such that
$(D-A)\cdot C\geq 0$ for all $\mov1$-curves $C$ on $X$.
If we apply this to the  $\mov1$-curve $C:=(D+\lambda A)^{n-2}\cdot E$,
then we obtain
$$
(D+\lambda A)^{n-1}\cdot E\geq A\cdot (D+\lambda A)^{n-2}\cdot E.
$$
Hence by the Khovanskii-Teissier inequality (Theorem \ref{thrm-hodge ineq}),
$$
(D+\lambda A)^{n-1}\cdot E\geq (A^{n-1}\cdot E)^{\frac{1}{n-1}}\cdot((D+\lambda A)^{n-1}\cdot E)^{\frac{n-2}{n-1}},
$$
and $(D+\lambda A)^{n-1}\cdot E\geq A^{n-1}\cdot E$.
This shows that by taking the limit $\lambda \to 0$, we get a contradiction
$$
0\leftarrow (D+\lambda A)^{n-1}\cdot E\geq A^{n-1}\cdot E>0.
$$\QED
\end{proof}

We give a partial affirmative answer to the problem posed in \cite{payne}
for the following type of varieties.
\begin{definition}\label{def-FT var}
A $\Q$-factorial variety $X$ is said to be of \emph{Fano type(FT)} if there exists a boundary
$\Q$-divisor $B$ on $X$ such that $(X,B)$ is klt, $K+B\equiv 0$ and the divisors
in the support of $B$ generate $\Nu^1(X)$.
\end{definition}
See \cite[Lemma-Definition 2.8]{prsh} for equivalent definitions.

\begin{corollary}\label{cor-FT mov}
For a $\Q$-factorial  FT variety $X$, the following duality holds:
$$
\ol{\Mob}(X)^{\vee}=\ol{\bNM}^1(X).
$$
Furthermore, the cones $\ol{\Mob}(X)$ and $\ol{\bNM}^1(X)$ are
closed convex and rational polyhedral.
\end{corollary}

\begin{proof}
There exists an effective boundary $\Q$-divisor $B$ such that
$(X,B)$ is klt, $K+B\equiv 0$ and the components of $B$ generate $\Nu^1(X)$.
There exists an effective ample divisor $A$ such that $\Supp A=\Supp B$.
The pair $(X,B-\varepsilon A)$ is klt for sufficiently small $\varepsilon>0$
and $-(K+B-\varepsilon A)$ is ample.
Therefore, the cone $\ol\NE(X)$ is $(K+B-\eps A)$-negative and
the equality follows immediately from
Theorem \ref{thrm-mob bnm dual} and Theorem \ref{thrm-cone thrm NM^1}.
The last statement follows from the rational polyhedral
property and the finiteness of the decomposition in Theorem \ref{thrm-bchm decomposition}.
See also \cite[Corollary 4.5]{chsh}. \QED
\end{proof}

\begin{remark}
In \cite{payne}, it is shown that  for complete $\Q$-factorial toric varieties $X$ and $0\leq k\leq \dim X$,
the duality holds between the closed cone in $\Nu^1(X)$ spanned by divisors that are
ample in codimension $k$ and
the closed cone in $\Num1(X)$ spanned by the curves that are birationally movable in codimension $k$.
(see \cite[Theorem 2]{payne}).
Payne asks if this is also true for general non-toric varieties.
Note that the two extreme cases $k=0$ and $k=\dim X$ are
true by Theorem \ref{thrm-NM Eff dual}.
Corollary \ref{cor-FT mov} gives an affirmative answer to this problem
for $\Q$-factorial FT varieties $X$ for the case $k=\dim X-1$.
It is also easy to see from the proof of Theorem \ref{thrm-mob bnm dual}
that the same duality holds for Mori dream spaces \cite{YS}.
Indeed, the duality holds in the part of the cone $\ol\Mob(X)$
where we can run the MMP.
In \cite{payne}, it is also explained that considering only the $\mov1$-curves in Theorem \ref{thrm-mob bnm dual}
is not enough in general (see \cite[Example 1]{payne}).
\end{remark}

\section{Cone theorems}\label{section-cone theorem}
Inspired by the results in \cite{araujo} and \cite{leh},
we prove the cone theorem (Theorem \ref{thrm-cone thrm NM^1}) and the contraction theorem
(Theorem \ref{thrm-contraction thrm NM^1}) for the cone $\ol\bNM^1(X)$
in this section.
\ms

Let $(X,B)$ be a $\Q$-factorial klt pair.
In the space $\Num1(X)$, we consider the following two convex cones:
(see Figure \ref{figure})
$$
\begin{array}{rl}
V(X,B)(=V)&:=\ol\NE(X)_{K+B\geq 0}+\ol{\bNM}^1(X),\\
V'(X,B)(=V')&:=\ol\NE(X)_{K+B\geq 0}+\ol{\NM}(X).\\
\end{array}
$$

An extremal face $F$ of  $\ol{\bNM}^1(X)$ is called a \emph{$\mov1$-co-extremal face}
for the pair $(X,B)$ if $F$ is a $(K+B)$-negative extremal face of $V$.
A divisor $D$ which is positive on $\ol\NE(X)_{K+B\geq 0}\setminus\{0\}$ and such that
the plane $\{\alpha\in\Num1(X)|\alpha\cdot D=0\}$ supports
the cone $V$ exactly at a $\mov1$-co-extremal face $F$ is called a \emph{$\mov1$-co-bounding divisor} of $F$.
An extremal face $F'$ of  $\ol\NM(X)$ is called a \emph{co-extremal face}
for the pair $(X,B)$ if $F'$ is a $(K+B)$-negative extremal face of $V'$ \cite{baty}.
A divisor $D$ which is positive on $\ol\NE(X)_{K+B\geq 0}\setminus\{0\}$ and
such that the plane $\{\alpha\in\Num1(X)|\alpha\cdot D=0\}$
supports the cone $V'$ exactly at a co-extremal face $F'$ is called a
\emph{co-bounding divisor} of $F'$.

\bs

 \begin{table}[htp]
  \begin{center}
    \begin{tabular}{ l c r }
\begin{overpic}[scale=0.46]{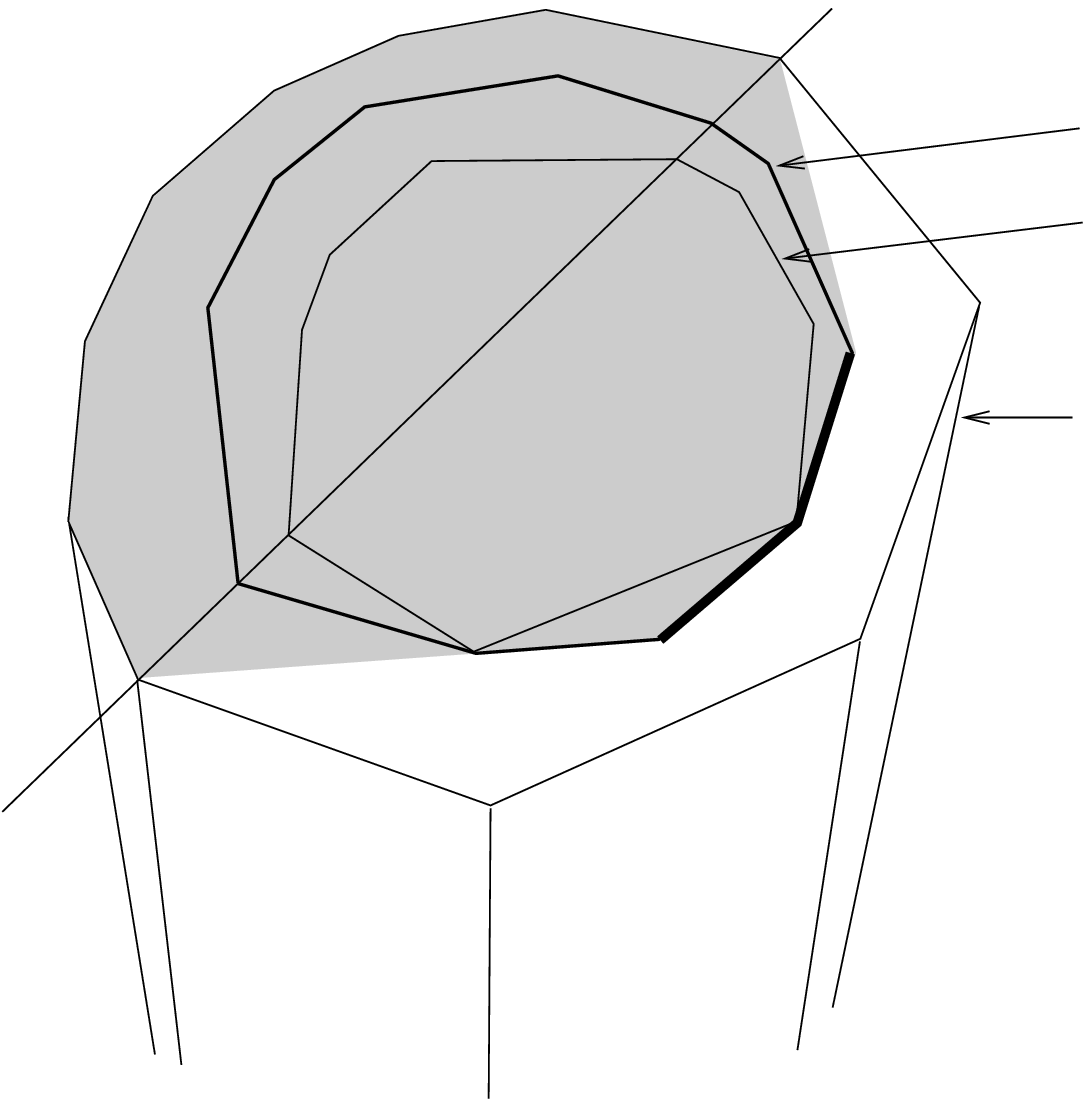}
\put(101,88){\tiny $\ol{\bNM}^1(X)$}
\put(101,61){\tiny $\ol{\NE}(X)$}
\put(75,101){\tiny $K+B=0$}
\put(15,100){\tiny $V_{K+B>0}$}
\put(0,-8){\tiny $V=\ol{\NE}(X)_{K+B\geq0}+\ol{\bNM}^1(X)$}
\end{overpic} & \;\;\;\;
 &\begin{overpic}[scale=0.46]{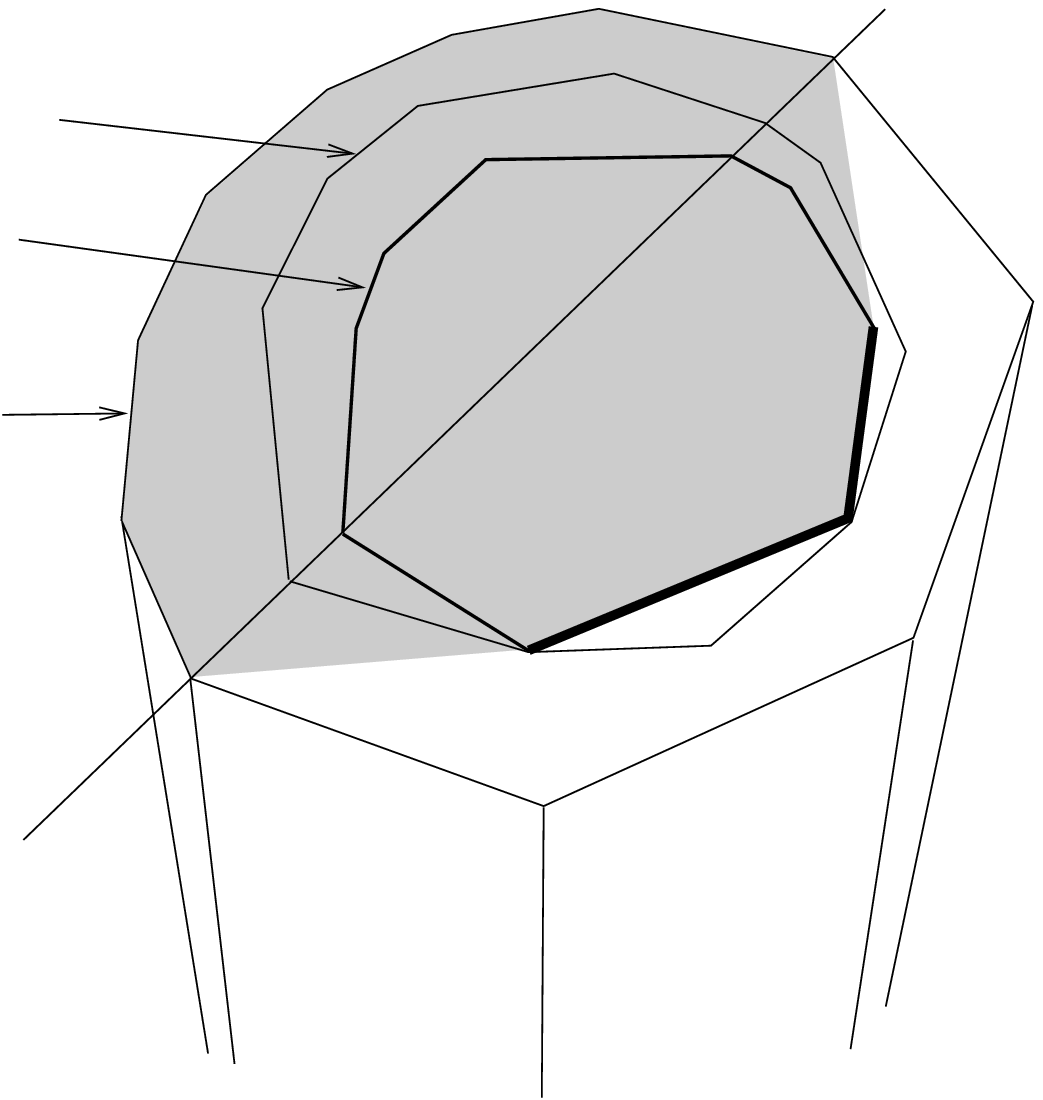}
\put(78,102){\tiny $K+B=0$}
\put(18,100){\tiny $V'_{K+B>0}$}
\put(-19,79){\tiny $\ol{\NM}(X)$}
\put(10,-8){\tiny $V'=\ol{\NE}(X)_{K+B\geq0}+\ol{\NM}(X)$}
\end{overpic} \\
    \end{tabular}
  \end{center}
  \captionsetup{type=figure}
  \setlength{\abovecaptionskip}{20pt}
  \caption{}\label{figure}
\end{table}

As illustrated in Figure \ref{figure}, an extremal face of $\ol{\bNM}^1(X)$ (resp. $\ol{\NM}(X)$)
in $\ol\NE_{K+B<0}(X)$ is not necessarily a
$\mov1$-co-extremal face (resp. a co-extremal face).
Note also that a co-extremal face of $\ol\NM(X)$ can coincide with a
$\mov1$-co-extremal face of $\ol{\bNM}^1(X)$.
It is also easy to see that the cone $\ol{\bNM}^1(X)$ has a $\mov1$-co-extremal ray for $(X,B)$
if and only if $K+B\not\in\ol\Mob(X)$ by Theorem \ref{thrm-mob bnm dual}
and the cone $\ol\NM(X)$ has a co-extremal  ray for $(X,B)$ if and only if $K+B\not\in\ol\Eff(X)$
by (2) of Theorem \ref{thrm-NM Eff dual}.

\bs

We have the following cone theorem for $\ol{\NM}(X)$ and
the contraction theorem for co-extremal rays (\cite{araujo},\cite{leh}).

\begin{theorem}[Cone Theorem for $\ol\NM(X)$]\label{thrm-cone thrm NM}
\emph{\cite[Theorem 1.1]{araujo}, \cite[Theorem 1.3]{leh}}
Let $(X,B)$ be a dlt pair. There are countably many $(K+B)$-negative movable curves
$\{C_i\}_{i\in I}$ such that
$$
\ol\NE(X)_{K+B\geq 0}+\ol\NM(X)=\ol\NE(X)_{K+B\geq 0}+\ol{\sum_{i\in I}\R_{\geq 0}\cdot[C_i]}.
$$
The rays $\R_{\geq 0}\cdot[C_i]$ can accumulate only along the hyperplanes supporting both
$\ol{\NM}(X)$ and $\ol\NE(X)_{K+B\geq 0}$.
\end{theorem}

\begin{remark}\label{remk-cone them NM}
As explained in \cite[Example 1.4]{araujo} and \cite[Example 4.9]{leh},
the genuine form of the cone theorem does not hold for $\ol{\NM}(X)$ in general, that is,
we cannot replace $\ol\NE(X)$ in Theorem \ref{thrm-cone thrm NM} by
$\ol{\NM}(X)$.
\end{remark}

\begin{theorem}[Contraction theorem for co-extremal faces]\label{thrm-contraction thrm NM}
\emph{\cite[Theorem 1.4]{leh}}
Let $(X,B)$ be a dlt pair. Suppose that $F'$ is a co-extremal face of $\ol\NM(X)$ for $(X,B)$
and $D$ be a co-bounding divisor of $F'$. Then there exists a birational morphism
$\varphi:W\rightarrow X$ and a contraction $h:W\rightarrow Z$ such that
\begin{enumerate}
\item[(1)] Every movable curve $C$ on $W$ with $[\varphi_*C]\in F'$ is contracted by $h$.
\item[(2)] For a general pair of points in a general fiber of $h$, there is a movable curve $C$ through the two points with $[\varphi_*(C)]\in F'$.
\end{enumerate}
These properties determine the pair $(W,h)$, up to a birational equivalence. In fact, the map we construct satisfies a stronger property:
\begin{enumerate}
\item[(3)] There is an open set $U\subseteq W$ such that the complement of $U$ has codimension $2$ in a general fiber of $h$ and a complete curve $C$ in $U$ is contracted by $h$ if and only if $[\varphi_*C]\in F'$.
\end{enumerate}
\end{theorem}

\begin{remark}\label{remk-nonexposed ray}
If $K+B\in\partial\ol\Eff(X)$, then $\ol{\NM}(X)\subseteq\ol\NE(X)_{K+B\geq 0}$
by Theorem \ref{thrm-NM Eff dual}. Thus there are no co-extremal faces for the pair $(X,B)$.
However, there exists an extremal face $F'$ of $\ol{\NM}(X)$ in $\ol{\NE}(X)_{K+B=0}$.
If $B$ is big, then $K+B-\eps B$ is not pseudoeffective
for any $\eps>0$ because $F'$ is $(K+(1-\eps)B)$-negative.
Thus there exists a co-extremal ray of $\ol\NM(X)$ for the pair $(X,(1-\eps)B)$ and
since $(X,(1-\eps)B)$ is klt for small $\eps>0$, the above theorems can be applied
to this pair. In particular, some extremal rays of $F'$ are contractible on some birational model of $X$.
\end{remark}

Theorem \ref{thrm-cone thrm NM^1} and Theorem \ref{thrm-contraction thrm NM^1}
can be considered as another analogues of the original Cone Theorem
and Contraction Theorem for $\ol\NE(X)$ (e.g. \cite[Theorem 3.7]{komo}).
We closely follow the paper \cite{leh}
in the proofs below. We start with a lemma.

\begin{lemma}\label{lem-quasibounding div}
Let $(X,B)$ be a $\Q$-factorial klt pair.
Let $K+B\not\in\ol{\Mob}(X)$ and
$F$ be a $\mov1$-co-extremal face of $\ol{\bNM}^1(X)$ for $(X,B)$.
If $D$ is a $\mov1$-co-bounding divisor of $F$,
then there exists an ample divisor $H$ such that $K+B+H$ is ample and $\alpha D\equiv K+B+cH$
for some $\alpha>0$ and $0<c<1$.
\end{lemma}
\begin{proof}
Let $G$ be the $2$-dimensional closed convex cone in $\Nu^1(X)$ spanned by $D$ and $-(K+B)$.
It is enough to prove that the $\emptyset\neq G\cap\Amp(X)$.
Indeed, a sufficiently large ample divisor $H\in G\cap\Amp(X)$ satisfies the required conditions.

Suppose that $G\cap \Amp(X)=\emptyset$. Then there exists a curve class $L$ which separates
the two cones: $L\cdot D'<0$ for all $D'\in G\setminus\{0\}$ and $L\cdot D''>0$ for all
$D''\in \Amp(X)$.
By the second inequality, $L\in\ol\NE(X)$. The first inequality with
$D'=-(K+B)$ gives $L\in\ol\NE(X)_{K+B>0}$.
However the first inequality with $D'=D$ also gives $L\cdot D<0$, contradicting the fact that
$D$ is positive on $\ol\NE(X)_{K+B>0}$.\QED
\end{proof}

Conversely, it is easy to see that the divisors $D\in\partial\ol{\Mob}(X)$ of the form
$D\equiv K+B+H$ for an ample divisor $H$ are $\mov1$-co-bounding divisors of some $\mov1$-co-extremal face.

\begin{proposition}\label{prop-finite NM^1}
Let $(X,B)$ be a $\Q$-factorial klt pair.
Consider the cone
$$V=\ol\NE(X)_{K+B+H\geq 0}+\ol{\bNM}^1(X)$$ for some ample divisor $H$
such that $(X,B+H)$ is klt.
Then there exists a finite set $\{C_i\}$ of $\bmov1$-curves of $X$ such that
for any $\mov1$-co-bounding divisor $D$ for some $\mov1$-co-extremal face of the cone $\ol{\bNM}^1(X)$ for
$(X,B+H)$, $[C_i]\cdot [D]=0$ for some $C_i$.
\end{proposition}
\begin{proof}
We may assume that $K+B+H\not\in\ol\Mob(X)$.
Otherwise, $K+B+H$ is nonnegative on $\ol\bNM^1(X)$ by Theorem \ref{thrm-mob bnm dual}
and there would be no $\mov1$-co-extremal rays.
Let $D$ be a $\mov1$-co-bounding divisor as in the statement.
Then by Lemma \ref{lem-quasibounding div},
there exists an ample divisor $A$ such that $K+B+H+A$ is ample and
$\alpha D\equiv K+B+H+cA$ for $\alpha>0$ and $0<c<1$. By Lemma \ref{lem-ample klt},
we may assume that the pair $(X,B+H+A)$ is klt.
By Theorem \ref{thrm-from bchm}, there exist a log terminal model $\varphi:X\dashrightarrow X'$
of $(X,B+H+cA)$, which is an isomorphism in codimension $1$ since $K+B+H+cA\in\partial\ol\Mob(X)$.

There also exists a contraction $\psi:X'\rightarrow Y$ which is either
a birational morphism contracting a divisor
or has a Mori fiber space structure where $K_{X'}+B_{X'}+H_{X'}+cA_{X'}$
vanishes on every curve contracted by $\psi$.
If $K+B+H+cA$ is big, then
there exists a divisorial component $E\subseteq\mathbf B_+(K_{X'}+B_{X'}+H_{X'}+cA_{X'})$
by Lemma \ref{lem-D in baoundary Mob} and Proposition \ref{prop-component under small}.
By \cite[Proposition 1.5]{bbp}, the divisor $E$ is $\psi$-exceptional
and there exists a $\mov1$-curve $C'$ on $X'$
contracted by $\psi$.
If $K+B+H+cA\in\partial\ol\Eff(X)$, then $\psi$ has a Mori fiber space structure and
a movable curve $C'$ is contracted $\psi$.
Thus, in either case, we obtain a $\bmov1$-curve $C'$ of $X$ such that
$[C']\cdot [D]=0$.
By the finiteness (Theorem \ref{thrm-bchm decomposition}),
as we vary $D$ in $\mathcal E_H$ such that $D\in\partial\ol\Mob(X)$,
we obtain only finitely many maps $\psi\circ\varphi$ and consequently finitely many
$\bmov1$-curves.\QED
\end{proof}

\begin{proof}[Proof of Theorem \ref{thrm-cone thrm NM^1}]
We may assume that $K+B\not\in\ol\Mob(X)$.
Otherwise, $K+B$ is nonnegative on $\ol\bNM^1(X)$ by Theorem \ref{thrm-mob bnm dual}
and there would be no $\mov1$-co-extremal rays.
We may assume that $(X,B+H)$ is klt by Lemma \ref{lem-ample klt}.
Let $\{\varepsilon_j\}$ be a sequence of strictly decreasing positive numbers converging to $0$.
Let $\{C_{ji}\}_{i\in I_j}$ be the finite set of all $\bmov1$-curves for $(X,B+\varepsilon_j H)$
obtained as in Proposition \ref{prop-finite NM^1} using Theorem \ref{thrm-from bchm} .
Then clearly,
$$
\ol\NE(X)_{K+B+\eps_jH\geq0}+\ol{\bNM}^1(X)\supseteq
\ol{\NE}(X)_{K+B+\eps_jH\geq0}+\ol{\sum_{i\in I_j}\R_{\geq0}\cdot[C_{ji}]}.
$$
Suppose that the strict inclusion $\supsetneq$ holds.
Then there exists a $\mov1$-co-extremal ray $R$ for $(X,B+\eps_j H)$ such that
$R\setminus\{0\}$ is disjoint from
$\ol{\NE}(X)_{K+B+\eps_jH\geq0}+\ol{\sum_{i\in I_j}\R_{\geq0}\cdot[C_{ji}]}$.
If $D$ is a $\mov1$-co-bounding divisor of $R$,
then by Lemma \ref{lem-quasibounding div},
there exists an ample divisor $A$ such that $K+B+\eps_jH+A$ is ample
and $\alpha D\equiv K+B+\eps_jH+cA$ for $\alpha>0$ and $0<c<1$.
Since we may assume that
$(X,B+\eps_jH+cA)$ is klt, by applying Theorem \ref{thrm-from bchm} on $(X,B+\eps_jH+cA)$,
we obtain a $\bmov1$-curve $C$ of $X$
(as in the proof of Proposition \ref{prop-finite NM^1}) such that $R=\R_{\geq 0}\cdot[C]$.
Since $R\not\in\{\R_{\geq 0}\cdot [C_{ji}]\}$
and $R$ is also a $\mov1$-co-extremal ray for $(X,B+\eps_j H)$,
it is a contradiction. So the second equality holds.

Suppose that the set $\cup_{j\in \mathbb N} I_{j}$ is infinite.
Since $I_j\subseteq I_{j+1}$ for all $j$, we may assume that $I_j\subsetneq I_{j+1}$ for all $j$.
By taking the limit $j\rightarrow \infty$, we obtain the second equality of the cones
and the last statement.\QED
\end{proof}

\bs

\begin{proof}[Proof of Theorem \ref{thrm-contraction thrm NM^1}]
We may assume that $K+B\not\in\ol\Mob(X)$.
Otherwise, $K+B$ is nonnegative on $\ol\bNM^1(X)$ by Theorem \ref{thrm-mob bnm dual}
and there would be no $\mov1$-co-extremal rays.
For a fixed $\mov1$-co-extremal ray $R$, by Lemma \ref{lem-quasibounding div}, there exists an ample divisor $H$ such that
$K+B+H$ is ample and $D=K+B+cH$ ($0<c<1$) is
a $\mov1$-co-bounding divisor for $R$.
We may assume that $(X,B+H)$ is klt by Lemma \ref{lem-ample klt}.
By Theorem \ref{thrm-from bchm}, there exists a log terminal model
$\varphi:X\dashrightarrow X'$ of $(X,B+cH)$ and
since $D\equiv K+B+cH\in\partial\ol{\Mob}(X)$,
the birational map $\varphi:X\dashrightarrow X'$ is an isomorphism in codimension $1$.

If $D\in\partial\ol\Eff(X)$, then the ray $R$ is a co-extremal ray of $\ol\NM(X)$ for $(X,B)$.
By \cite{bchm}, there exists a Mori fiber space structure
$X'\rightarrow Y$ and
the statements follow from Theorem \ref{thrm-contraction thrm NM}.
Assume that $D\in\Int\ol\Eff(X)$.
Then as in the proof of Proposition \ref{prop-finite NM^1},
the ray $R$ is spanned by a $\mov1$-curve $C'$ on $X'$
(which is not movable and is a $\bmov1$-curve of $X$) and
its associated contraction $\psi:X'\rightarrow Y$ is divisorial.
Now by the uniqueness of the (lc) model $Y$ (Theorem \ref{thrm-from bchm})
for the $\mov1$-co-bounding divisor $D=K+B+cH$ of $R$,
we obtain the statements (1) and (2).
\QED
\end{proof}

\begin{remark}
As illustrated in the Figure \ref{figure}, there may be an extremal ray $R$
of $\ol{\bNM}^1(X)$ which is not $\mov1$-co-extremal, but co-extremal.
This ray does not appear in the expression
$\ol\NE_{K+B\geq 0}(X)+\ol{\sum_{i\in I}\R_{\geq 0}\cdot[C_i]}$ in Theorem \ref{thrm-cone thrm NM^1}.
However,  the statements of Theorem \ref{thrm-contraction thrm NM^1}
also hold for this ray by Theorem \ref{thrm-contraction thrm NM}.
\end{remark}

In the statements of Theorem \ref{thrm-contraction thrm NM^1}, if $K+B$ is not big, then
$\psi:X'\rightarrow Y$ is a Mori fibration and this is a resulting model of the given pair $(X,B)$.
Note that if $K+B\in\partial\ol\Eff(X)$, then $K_{X'}+B_{X'}$ is $\psi$-trivial and $Y$
is the lc Iitaka model of $(X,B)$.
If $K+B$ is big and $K+B\in\partial\ol{\Mob}(X)$, then $(X',B_{X'})$ is a resulting model
which is a log terminal model of $(X,B)$
and the contraction $\psi:X'\rightarrow Y$ is the lc contraction to the lc model $Y=X\cn$ of
$(X,B)$. For all other cases, namely, when $K+B$ is big but not in $\ol\Mob(X)$,
the divisorial contraction $\psi$ in Theorem \ref{thrm-contraction thrm NM^1} is only one of the intermediate modifications of the LMMP.
\bs

\begin{remark}
If $K+B\in\ol\Mob(X)$, then the cone $\ol{\bNM}^1(X)$
does not have any $\mov1$-co-extremal faces.
However, if $K+B\in\partial\ol\Mob(X)$, then $\ol{\bNM}^1(X)$
has extremal faces in  $\ol\NE(X)_{K+B=0}$.
If $K+B$ is big or $B\in\Int\ol\Mob(X)$, then some of such faces $F$ are
$\mov1$-co-extremal for some klt pair
and Theorem \ref{thrm-contraction thrm NM^1} holds for these rays too.
Indeed, suppose $\ol\bNM^1(X)\subseteq\ol\NE(X)_{K+B\geq 0}$ and
let $F$ be an extremal face of $\ol{\bNM}^1(X)$ in $\ol\NE(X)_{K+B=0}$.
If $K+B$ is big, then $K+B\equiv H+E$ for some ample $H$ and effective $E$.
For a small $\varepsilon>0$, $K+B+\varepsilon E$ is big and $(X,B+\varepsilon E)$ is still klt.
However, $K+B+\varepsilon E\not\in\ol{\Mob}(X)$ since
we can easily check that $F$ is $(K+B+\eps E)$-negative and
$\ol{\NE}(X)_{K+B+\varepsilon E=0}$ does not intersect with
the supporting plane $\{[C]\in\Num1(X)\;|\;C\cdot (K+B)=0\}$.
Therefore, $F$ is a $\mov1$-co-extremal face of $\ol{\bNM}^1(X)$ for the pair
$(X,B+\varepsilon E)$ and $K+B$ is a $\mov1$-co-bounding divisor for $F$.
Since the extremal rays of $F$ are $\mov1$-co-extremal rays,
Theorem \ref{thrm-cone thrm NM^1} and Theorem \ref{thrm-contraction thrm NM^1} can be applied to this case.
The similar argument works for the case when $B\in\Int\ol\Mob(X)$ (cf. Remark \ref{remk-nonexposed ray}).
\end{remark}

\bs

\begin{question}
In \cite{baty}, Batyrev conjectured that the co-extremal rays
in Theorem \ref{thrm-cone thrm NM} do not accumulate away from $\ol\NE(X)_{K+B=0}$.
Similarly, we can ask whether the $\mov1$-co-extremal rays in Theorem \ref{thrm-cone thrm NM^1}
can accumulate away from $\ol\NE(X)_{K+B=0}$.
For the results related to the conjecture of Batyrev or the cone $\ol{\NM}(X)$, see
\cite{araujo},\cite{barkow},\cite{baty},\cite{bchm},\cite{leh},\cite{qihong}.
\end{question}



\end{document}